\documentclass[10pt]{article}
\usepackage{amssymb,amsmath,amsthm,amsfonts,amscd}
\usepackage{mathrsfs}
\usepackage[all]{xy}
\usepackage[latin1]{inputenc}
\usepackage{auxiliar}
\numberwithin{equation}{section}

\begin{document}

\title{Regularity for elliptic pairs over $\C[[\h]]$} 
\author{David Raimundo
\footnote{The research of the author
was supported by Funda\c c{\~a}o para a Ci\^encia e Tecnologia and by Funda\c c\~ao Calouste Gulbenkian (Programa Est\'imulo \`a Investiga\c c\~ao).
}}
\maketitle

\begin{abstract} 
We extend the results of Schapira and Schneiders~\cite{ScS} on relative regularity, finiteness and duality (in the smooth case) of elliptic pairs 
to the framework of $\shd[[\hbar]]$-modules and constructible sheaves of $\C[[\h]]$-modules. 
\end{abstract}

\section*{Introduction} 
The notion of elliptic pair goes back to~\cite{ScS}, 
where the authors consider a morphism of complex analytic manifolds $f:X\to Y$ and say that 
a coherent $\shd_X$-module $\shm$ and an $\R$-constructible sheaf of $\C$-modules $F$ form an $f$-elliptic pair if 
the $f$-characteristic variety of $\shm$ and the micro-support of $F$ do 
not intersect outside the zero-section of the cotangent bundle $T^\ast X$. 
%The $f$-characteristic variety of a coherent $\shd_X$-module is a closed conic analytic subvariety of the cotangent bundle $T^\ast X$, 
%which depends on $f$. In particular, the $f$-characteristic variety coincides with the classical characteristic variety when $f$ is the constant map $X\to \rmptt$.

Consider the constant map $a_X:X\to\rmptt$. If $X$ is the complexification of a real analytic manifold $M$ and $\shm$ is an elliptic system on $M$, 
then $(\shm,\C_M)$ is an $a_X$-elliptic pair. Elliptic pairs can thus be regarded as a generalization of elliptic systems.
The functorial properties of elliptic pairs are studied in~\cite{ScS}, where theorems of 
regularity, finiteness and duality are proved for these objects. As the authors point out, such theorems generalize several classical results 
of $\shd_X$-modules theory, complex analytic geometry and elliptic systems theory. 

The purpose of this paper is to extend the main results of~\cite{ScS} to the framework 
of modules over the ring $\shd_X[[\h]]$, the ring of differential 
operators with a formal parameter $\h$. 
This ring appeared first in~\cite{KS3} as an example of an algebra of formal deformation. 
Therefore, the machinery developed in~\cite{KS3} to perform the study 
of deformation-quantization modules also apply to the study of $\shd[[\h]]$-modules. 
Such study has been performed in~\cite{DGS} and~\cite{MMFR}. 

Denote by $\C[[\h]]$ the ring of formal power series on $\C$. Consider the right exact functor 
that maps each sheaf of $\C[[\h]]$-modules $F$ to $F/\h F$ regarded as a sheaf of $\C$-modules.
Consider also its left derived functor denoted by $\gr$.
% in the derived category of sheaves of 
%$\C[[\h]]$-modules to $\gr(F)\eqdot F\lltens[{\C_X[[\h]]}]\C_X$ regarded as an object in the derived category 
%of sheaves of $\C$-modules. 

In this work we introduce the notion of $f$-elliptic pair over $\C[[\hbar]]$ in a natural way: 
if $\shm$ is a coherent $\shd_X[[\h]]$-module and $F$ is an $\R$-constructible sheaf of $\C[[\h]]$-modules, then $(\shm,F)$ is an $f$-elliptic pair over $\C[[\h]]$ if and only if $(\gr(\shm),\gr(F))$ is an elliptic pair in the sense of~\cite{ScS}. 
This allow us to prove theorems for $f$-elliptic pairs over $\C[[\h]]$ 
using properties of $\gr$ given by~\cite{KS3} to reduce the proofs to the theorems of~\cite{ScS}. 

Let us mention our main results. In Theorems~\ref{T53} and~\ref{T54} we prove regularity properties for $f$-elliptic pairs over $\C[[\h]]$. 
These regularity theorems generalize a classical regularity property of elliptic systems in the real analytic setting: 
the complex of real analytic solutions of an elliptic system is isomorphic to the complex of hyperfunctions solutions. 

In Theorem~\ref{T55} we use the regularity of $f$-elliptic pairs over $\C[[\h]]$ to give a finiteness criteria. 
Denote by $\deimh f$ the proper direct image in the framework of $\shd_X[[\h]]$-modules introduced in~\cite{MMFR}.
The statement of the theorem is the following: given an $f$-elliptic pair $(\shm,F)$, such that $f$ is proper when restricted to $\supp(\shm)\cap\supp(F)$ 
and such that $\shm$ is good, 
then the cohomologies of the direct image $\deimh f (\shm\lltens[{\shd_X[[\h]]}]F)$ are coherent over $\shd_Y[[\h]]$. 

In Theorem~\ref{T60} we prove a duality result in the case of a smooth morphism. 
It states that the direct image and the duality functor for $\shd_X[[\h]]$-modules 
commute when applied to $f$-elliptic pairs that satisfy the finiteness criteria.
The reason why we must restrict to the smooth case is that two fundamental properties hold in this case:
the transfer module is coherent over $\shd_X[[\h]]$ and 
the extension rings $f^{-1}(\shd_X[[\h]])$ and $(f^{-1}\shd_X)[[\h]]$ are isomorphic. 
These properties are necessary to our construction of the duality morphism. 
We note that the smooth case includes the interesting case $f=a_X$.

%More precisely, we prove that if $f$ is smooth, then it is possible to construct a morphism 
%\begin{equation}\tag{*}
%\deimh f(\RDD_{\h} F\lltens[{\shd_X[[\h]]}] \ul{\RD}_{\h,X} \shm) \to \ul{\RD}_{\h,Y} \deimh f(\shm\lltens[{\shd_X[[\h]]}] F). 
%\end{equation}
%Moreover, if $(\shm,F)$ is an $f$-elliptic pair over $\C[[\h]]$ in the conditions of the finiteness theorem, 
%then (*) is an isomorphism of complexes of $\shd_Y[[\h]]$-modules with coherent cohomologies (cf.~Theorem~\ref{T60}).

In the last part of the paper we illustrate our results in some particular cases. 
For example, in the case where $X$ is the complexification of a real analytic manifold $M$, 
we obtain regularity, finiteness and duality properties on the sheaves of formal analytic functions and formal hyperfunctions on $M$.

%In the last section, we study particular cases of the main theorems. Namely, 
%we obtain finiteness and duality results for the complex of global solutions of a coherent $\shd_X[[\h]]$-module and also a refinement in the holonomic case (cf. Proposition~\ref{coroA}). 
%We also state finiteness and duality theorems for $\sho_X^\h$-modules and discuss those theorems in view of our results.
%Finally, in the case where $X$ is the complexification of a real analytic manifold $M$, 
%we obtain regularity, finiteness and duality properties concerning the sheaves of formal analytic functions and formal hyperfunctions.

\paragraph{Acknowledgments.} 
We thank Pierre Schapira for proposing the subject 
of this paper and for making useful 
remarks and suggestions. 
We thank Teresa Monteiro Fernandes for her supervision, helpful discussions and constant incentive.
We also thank Ana Rita Martins for her comments.

%\section*{Notations and conventions} 

%We denote by $\rmptt$ the zero-dimensional manifold consisting of a point and by $a_X:X\to\rmptt$ the constant map. If $R$ is a noetherian ring, identified to a sheaf of rings over $\rmptt$, we write $\Derb_f(R)$ instead of $\Derb_\coh(R)$.

%\section{Preliminaries}
\section{Complements on formal extensions.}

In view of our purpose it suffices to work on the complex analytic setting,  
although some results hold in a more general situation.
In the sequel $X$ denotes a complex analytic manifold of dimension $d_X$.

%In what concerns the language of categories and sheaves,
We follow the notations of~\cite{KS1}. 
Namely, if $\shr$ is a sheaf of rings on $X$, we denote by $\md[\shr]$ 
the category of left $\shr$-modules and by $\Derb(\shr)$ the 
bounded derived category of $\md[\shr]$. 
If $\shr$ is coherent, 
we denote by $\mdc[\shr]$ the full abelian subcategory of $\md[\shr]$ 
of coherent objects and by $\Derb_\coh(\shr)$ 
the full triangulated subcategory
of $\Derb(\shr)$ of objects with coherent cohomology groups.
In the sequel, when the base ring is $\C$ we may omit it. 
%Namely, we shall write $F\tens G$ and $\rhom(F,G)$ instead of $F\tens[\C_X]G$ and $\rhom[\C_X](F,G)$, respectively.

%In this section we recall notions and results of~\cite{KS2} and~\cite{DGS} and 
%we prove some complementary results.

Let $\shro$ denote a sheaf of $\C_X$-algebras on $X$ 
and set $\shr\eqdot\shro[[\h]]=\Pi_{n\geq 0} \shro \h^n$. 
Then $\shr$ is a sheaf of algebras over the ring $\C[[\h]]$, the ring of formal power series with complex coefficients. 
One uses the abbreviation $\coro\eqdot\C[[\h]]$.

Consider the following left exact functor studied in~\cite{DGS}:
\eq\label{h}
(\scbul)^{\h}:\md[\shro] &\to& \md[\shr]\\ 
\shn &\to& \shn^{\h}\eqdot \shn[[\h]]
=\varprojlim_{n\geq 0}(\shn\tens[\shro]\shr/\h^{n+1}\shr).\nonumber
\eneq
Recall that the sections of $\shn^\h$ on an open subset of $X$ are formal 
power series of sections of $\shn$ on the same open subset.

One denotes by $(\scbul)^\rhb\cl \Derb(\shro) \to \Derb(\shr)$ the right derived functor of $(\scbul)^\hbar$.
For each $F\in\Derb(\shro)$, $F^\rhb$ is called the formal extension of $F$. 

One says that $F\in\md[\shro]$ is $\h$-acyclic if $F^\rhb\simeq F^\h$.

%The following result of~\cite{DGS} states a condition for an object to be $(\scbul)^\h$-acyclic:

\begin{proposition}[\cite{DGS}, Proposition~2.5]\label{chc5} 
Let $\shn\in\md[\shro]$ and suppose that $\mathcal{B}$ is either a basis of open subsets of $X$, or a basis of compact subsets, such that $H^j(S;\shn)=0$ for all $j>0$ and all $S\in\mathcal{B}$. Then, $\shn$ is $\hbar$-acyclic.
\end{proposition}

From now on assume that $\shro$ is an $\h$-acyclic $\C_X$-algebra. 

\begin{lemma}\label{Aux2} 
Let $f:X\to Y$ be a morphism of complex manifolds and assume that 
$(\opb f \shro)^\h\simeq \opb f\shr$. Then, for each $\shn\in \Derb(\opb f \shro)$, 
we have a canonical morphism in $\Derb(\shr)$:
\eqn
\reim f(\shn^\rhb) \to \reim f(\shn)^\rhb.
\eneqn
\end{lemma}
\begin{proof}
Using Proposition 2.1 of~\cite{DGS} and formula (2.6.26) of~\cite{KS1} we have in $\Derb(\shr)$: 
\begin{eqnarray*} 
\reim f(\shn^\rhb)&\simeq&\reim f \rhom[f^{-1}\shro](f^{-1}(\shr^\loc/\hbar\shr), \shn)\\ 
&\to& \rhom[\shro](\roim f \opb f(\shr^\loc/ \h \shr), \reim f \shn)\\
&\to& \rhom[\shro] (\shr^\loc/\h\shr, \reim f \shn) \simeq (\reim f \shn)^{\rhb}.
\end{eqnarray*}
\end{proof} 

%\paragraph{Graded modules.}

Consider also the right exact functor that maps each $\shm\in\md[\shr]$ into $\shm/\h\shm\in\md[\shro]$. 
We shall use its left derived functor
\eqn
\gr:\Derb(\shr)\to\Derb(\shro), \quad \shm\to \gr(\shm)\eqdot \shm\lltens[\shr]\shro
\eneqn
which was studied in detail in~\cite{KS3}.
Recall that $\gr$ commutes with tensor products,    
with $\rhom$ and also with direct images, 
proper direct images and inverse images of sheaves. 
Recall also that $\shm\in\md[\shr]$ has no $\hbar$-torsion 
if $\gr(\shm)$ is concentrated in degree $0$, that is, if $\h:\shm\to\shm$ is injective.

One says that $\shm\in\md[\shr]$ is $\hbar$-complete if the canonical morphism 
$\shm\to \varprojlim_n \shm/\hbar^{n+1}\shm$ is an isomorphism.
Set $\shr^\loc\eqdot\C((\h))_X\tens[\coro_X]\shr$, where $\C((\h))$ 
denotes the field of Laurent series with complex coefficients.  
One says that $\shm\in\Der(\shr)$ is cohomologically complete if 
$\rhom[\shr](\shr^\loc,\shm)=0$. 
The notions of $\h$-complete object and cohomologically complete object don't 
depend on the base ring. 
%Hence, if $\shm\in\Der(\shr)$, then 
%$\shm$ is cohomologically complete 
%if and only if it is cohomologically complete as an object in $\Der(\coro_X)$ . 

We refer~\cite{KS3} for a comprehensive study of cohomologically complete objects. 
Let us just mention some facts that we shall use in the sequel: 
$\gr$ is conservative when restricted to the triangulated subcategory of 
$\Derb(\shr)$ consisting of cohomologically complete objects (cf. Corollary 1.5.9 of~\cite{KS3}) and 
$\shn^\rhb$ is cohomologically complete for every $\shn\in\Derb(\shro)$ (cf. Proposition 1.2 of~\cite{DGS}). 
%These properties explain why cohomologically complete objects play a central role in the theory of~\cite{KS3} and~\cite{DGS}.
%We also refer to~\cite{KS3} for a comprehensive study of cohomologically complete objects. 

In view of Proposition~\ref{VB} below let us fix a morphism of complex analytic manifolds $f:X\to Y$ 
and let $\shro$ be an $\hbar$-acyclic $\coro$-algebra on $Y$. 

\begin{remark}\label{remark1}
There is a canonical morphism of sheaves of $\coro$-algebras $f^{-1}\shr=f^{-1}\shro^\h\to (f^{-1}\shro)^\h$ induced by the morphisms 
\eqn
f^{-1}\shro^\h \to f^{-1}(\shro\tens \coro_Y/\h^{n+1}\coro_Y)\simeq f^{-1}\shro\tens \coro_X/\h^{n+1} \coro_X,
\eneqn
and by the universal property of projective limits. 
Hence, there is a canonical functor $F\in\md[(f^{-1}\shro)^\h]\mapsto F\in\md[f^{-1}(\shro^\h)]$. 
\end{remark}

%\begin{assumption}\label{assumption dual1}
%There exists either a basis $\mathcal{B}$ of open subsets of $X$ or, assuming that $X$ is a locally compact topological space, a basis $\mathcal{B}$ of compact subsets, such that $H^j(S;%\shro)=0$ for all $j>0$ and all $S\in\mathcal{B}$.
%\end{assumption}

%We need some additional hypothesis to prove the next result:

\begin{assumption}\label{Assumption R0}
There exists a basis $\mathcal{B}$ either of open subsets of $Y$ or of compact subsets of $Y$ 
such that $H^j(S;\shro)=0$, for all $j>0$ and for all $S\in\mathcal{B}$.
\end{assumption}

%\begin{assumption}\label{Assumption topological}
%Let $X^\prime\times X$ be a product of topological spaces and denote by $f:X^\prime\times X\to X$ the canonical projection. 
%Assume that there exists a family $\left\{X_n\right\}$ of closed subsets of $X^\prime\times X$ such that: $X^\prime\times X=\cup X_n$, $X_n\subset \Int(X_{n+1})$ and $f_{|_{X_n}}$ is %proper with contractible fibers for all $n$. 
%\end{assumption}

%\begin{remark}
%Consider the case of a smooth morphism $f:X\to Y$ of complex analytic manifolds. 
%Then $f$ is locally a projection and Assumption~\ref{Assumption topological} is locally automatically fulfilled.
%\end{remark}

\begin{proposition}\label{VB} 
Suppose that $f:X\to Y$ is smooth and that $\shro$ verifies Assumption~\ref{Assumption R0}. Then $\opb f \shr \isoto (\opb f \shro)^h$.
\end{proposition}
\begin{proof}
Consider the canonical morphism from Remark~\ref{remark1}. 
Note that the corresponding canonical morphism $\gr(\opb f \shr)\to\gr((\opb f \shro)^\h)$ is an isomorphism, since one has $\gr(\opb f \shr)\simeq\gr((\opb f \shro)^\h)\simeq \opb f \shro$.
Hence, in view of~\cite[Cor. 1.5.9]{KS3}, it suffices to show that both $\opb f\shr$ and $(\opb f \shro)^\h$ are cohomologically complete objects.

The ring $\shr\simeq\shro^\rhb$ is cohomologically complete. 
Since $f$ is smooth, $\shr$ is non-characteristic for $f$ and $\opb f \shr$ is cohomologically complete by Proposition 1.14 of~\cite{MMFR}. 

We shall use Proposition~\ref{chc5} to prove that $(\opb f\shro)^\rhb$ is concentrated in degree 0, 
thus $(\opb f\shro)^\h$ is cohomologically complete.

The result is now checked locally, so we can assume that $X=X^\prime \times Y$ and that $f:X\to Y$ is the canonical projection.

Let us consider the case where $\mathcal{B}$ is a basis of open subsets of $Y$. 
It is enough to show that there exists a basis $\mathcal{B}^\prime$ of open subsets of $X^\prime\times Y$ such that $H^j(S^\prime,\opb f \shro)=0$ for all $j>0$ and all $S^\prime\in\mathcal{B}^\prime$. 

Consider the basis $\mathcal{B}^\prime$ formed by the open sets $V^\prime\times V\subset X^\prime\times Y$ 
such that $V^\prime$ is an open ball of $X^\prime$ (hence, contractible) and $V\in\mathcal{B}$. 
We are in the conditions to apply 
%the variant of Vietoris-Begle theorem provided by
~\cite[Proposition 2.7.8]{KS1}.
Hence, for any $j>0$ one has: 
\eqn
H^j(V^\prime\times V,f^{-1}\shro)&\simeq &H^j(f_{|_{V^\prime\times V}}^{-1}(V),f_{|_{V^\prime\times V}}^{-1}\shro|_V) \simeq H^j(V,\shro|_V)=0.
\eneqn

The case where $\mathcal{B}$ is a basis of compact subsets of $X$ is similar, taking closed balls on $X^\prime$ instead of open balls. 
\end{proof}

\paragraph{$\R$-constructible sheaves of $\C[[\h]$-modules.}
As usual denote by $\SSi(F)$ the micro-support of an object $F\in\Derb(\coro_X)$, 
a closed involutive subset of the cotangent bundle $T^\ast X$. 
One has the estimative $\SSi(\gr(F))\subset\SSi(F)$. 
The equality $\SSi(\gr(F))=\SSi(F)$ holds if $F$ is 
cohomologically complete (cf. Proposition 1.15 of~\cite{DGS}), then in such case one also has $\supp(F)=\supp(\gr(F))$. 

Let $\cora$ be either $\C$ or $\coro$. 
One denotes by $\Derb_\Rc(\cora_X)$ 
the bounded derived category of sheaves of $\cora$-modules 
with $\R$-constructible cohomology. 
%See~\cite{KS1} for the general theory of constructible sheaves. 

Consider the constant map $a_X:X\to\rmptt$ and denote by $\omega_X^\cora$ the dualizing sheaf in the category 
$\Derb(\cora_X)$. Recall that one has $\omega_X^\cora\simeq \epb{a_X}\cora \simeq \cora_X[2d_X]$. 
%(cf \cite[Def.3.3.3]{KS1}). 
%Since every real manifold is oriented, one has 
%$\omega_X^\h\simeq \epb{a_X}\cora \simeq \cora_X[r]$. 
%where the shift $2d_X$ corresponds to the real dimension of the manifold.

We shall use the duality functors below: 
\eqn
&&\RDD_{\cora_X}:\Derb(\cora_X)\to \Derb(\cora_X), \quad F\mapsto \rhom[\cora_X](F,\cora_X),\\
&&\RD_{\cora_X}:\Derb(\cora_X)\to\Derb(\cora_X), \quad F\mapsto \rhom[\cora_X](F,\omega_X^\cora).
\eneqn
They induce functors $\RDD_{\cora_X},\RD_{\cora_X}:\Derb_\Rc(\cora_X)\to\Derb_\Rc(\cora_X)$ which satisfy the microlocal relation: 
$\SSi(\RDD_{\cora_X} F)=\SSi(\RD_{\cora_X} F)=\SSi(F)^a$, where $a$ denotes the opposite map on $T^\ast X$.

Any $F\in\Derb_\Rc(\coro_X)$ is cohomologically complete by Proposition 1.6 of~\cite{DGS}. 
Hence, $\gr:\Derb_\Rc(\coro_X)\to\Derb(\C_X)$ is conservative and preserves the micro-support. 
It is also known that $\R$-constructible sheaves are $\h$-acyclic (Corollary 2.6 of~\cite{DGS}).

The next proposition shows that 
$F^\h\tens[\coro_X]\scbul$ is an exact functor in $\md[\coro_X]$ for each $F\in\mdrc[\C_X]$.

\begin{proposition}\label{P29}
Let $F \in \Derb_\Rc(\C_X)$. 
Then we have $F^\rhb \simeq F \tens \coro_X$. 
\end{proposition}
\begin{proof}
First we remark that by replacing $F$ with an almost free resolution of $F$ (in the sense of the Appendix of~\cite{KS2}) 
one easily reduces the proof to the case $F=\C_U$, $U$ being an open subanalytic relatively compact subset of $X$. 
This reduction uses the fact that $\R$-constructible sheaves are $\hbar$-acyclic. 

%Let 
%\eqn
%F^\scbul: 0 \to \bigoplus_{i \in I_a} \C_{U_{a,i}} \to \cdots 
%\to \bigoplus_{i \in I_b} \C_{U_{b,i}} \to 0,
%\eneqn
%be an almost free resolution of $F$ in the sense of the Appendix of \cite{KS2}. 
%That means that $F^\scbul$ is quasi-isomorphic to $F$, each family $\left\{U_{k,i}\right\}_{k\in I_k}$ is locally finite 
%and each open subset $U_{k,i}$ of $X$ is subanalytic and relatively compact.

%Denote by $F^{\scbul,\hbar}$  the complex: 
%\eqn
%F^{\scbul,\hbar}: 0 \to \bigoplus_{i \in I_a} \coro_{U_{a,i}} \to \cdots 
%\to \bigoplus_{i \in I_b} \coro_{U_{b,i}} \to 0.
%\eneqn
%Since $\R$-constructible sheaves are $\hbar$-acyclic, 
%we have the following chain of quasi-isomorphisms in $\Derb(\coro_X)$:
%\eqn
%F^\rhb \simeq {(F^\scbul)}^\rhb \simeq F^{\scbul,\hbar}.
%\eneqn

%On the other hand, we have
%$F \tens \coro_X \simeq F^\scbul \tens \coro_X$ 
%where we set:
%\eqn
%F^\scbul \tens \coro_X: 0 \to \bigoplus_{i \in I_a} \C_{U_{a,i}} \tens \coro_X \to 
%\cdots \to \bigoplus_{i \in I_b} \C_{U_{b,i}} \tens \coro_X \to 0.
%\eneqn

Therefore, it is enough to note that 
for each open subanalytic subset $U$ we have a canonical isomorphism:
\eqn
\C_U \tens \coro_X \isoto \coro_U.
\eneqn
In fact, the stalks $(\C_U \tens \coro_X)_x$ and $(\coro_U)_x$ 
are both isomorphic to $\coro$ if $x \in U$, and both vanish if $x\notin U$.
\end{proof}

As a consequence of Proposition~\ref{P29} above and Lemma 1.9 of~\cite{DGS}, we have:

\begin{corollary}\label{T31} 
For $F \in \Derb_\Rc(\C_X)$ and $G \in \Derb(\C_X)$, 
there are isomorphisms in $\Derb(\coro_X)$:
\eqn\label{rhom+rhb1}
&(i)& \rhom[\coro_X](F^\rhb,G^\rhb) \simeq (\rhom(F,G))^\rhb\\
&(ii)& F^\rhb\lltens[\coro_X]G^\rhb \simeq F\tens G^\rhb.
\eneqn
\end{corollary}
%\begin{proof} 
%Applying Proposition~\ref{P29}
%we get the following chain of isomorphisms:
%\eqn
%\rhom[\coro_X](F^\rhb,G^\rhb) 
%&\simeq& \rhom[\coro_X](F \tens\coro_X,G^\rhb) \\
%&\simeq& \rhom(F,G^\rhb).
%\eneqn
%Then, the isomorphism (i) follows from Lemma~\ref{L15}.

%The proof of (ii) is similar. 
%\end{proof}

%Note that, as a particular case of (i), one gets $\RDD_{\coro_X} F^\h\simeq(\RDD F)^\h$ 
%for any $F\in\Derb_\Rc(\coro_X)$.

\section{Elliptic pairs over $\C[[\h]]$}

\paragraph{$\shd[[\h]]$-modules.}

Let $\sho_X$, $\Omega_X$ and $\shd_X$ denote the ring of holomorphic functions on $X$, 
the sheaf of holomorphic forms of maximal degree on $X$, and  
the ring of linear differential holomorphic operators on $X$, respectively.

As usual, $\mdgd[\D]$ denotes the full subcategory of 
$\mdc[\D]$ consisting of good $\D$-modules (in the sense of~\cite{Ka2})
and $\Derb_\gd(\D)$ denotes the full subcategory of $\Derb_\coh(\D)$
consisting of objects with good cohomology. 

An object $\shm\in\Derb_\coh(\Dh)$ is said to be good if 
$\gr(\shm)\in \Derb_\gd(\D)$. 
Denote by $\Derb_\gd(\Dh)$ the full triangulated subcategory 
of $\Derb_\coh(\Dh)$ consisting of good objects.

We shall use the duality functors below: 
\eqn
&&\ul{\RD}_{\h,X}:\Derb({(\Dh)}^\rop)\to\Derb({(\Dh)}^\rop),\quad  
\shm\mapsto\rhom[\Dh](\shm,\Omega_X[d_X] \tens[\sho_X] \Dh) \\
&&\RDD_{\h,X}:\Derb(\Dh)\to \Derb({(\Dh)}^\rop), \quad 
\shm\mapsto\rhom[\Dh](\shm,\shd_X).
\eneqn
Denote also by $\ul{\RD}_X$ and $\RDD_X$ their counterparts for $\shd$-modules.

The rings $\shd^\h_X$ and ${\shd_X^\h}^\rop$ are algebras of formal deformation in the sense of~\cite{KS3} 
and the machinery developed in loc. cit. apply to the study of $\shd_X^\h$-modules. 
In particular, objects belonging to $\Derb_\coh(\shd_X^\h)$ are cohomologically complete and 
Theorem 1.6.4 of~\cite{KS3} provide a useful coherence criteria for $\shd_X^\h$-modules. 
Recall also that coherent $\shd_X$-modules are $\hbar$-acyclic (cf. Corollary 2.6 of~\cite{DGS}).

In the sequel most of our results are stated for right 
$\shd_X^\hbar$-modules but one can 
get similar results for left $\shd_X^\hbar$-modules. 
Indeed, the category of left $\shd_X^\h$-modules and the category of right $\shd_X^\h$-modules 
are equivalent. Hence, we don't need to distinguish left and right $\shd^\h$-modules.

%\begin{remark}
%The analogous of Corollary~\ref{T31} can be stated for coherent 
%$\Dh$-modules by using Proposition 2.8 of~\cite{DGS}. 
%\end{remark}

\paragraph{Direct images.}
From now on, $f:X\to Y$ denotes a complex morphism between 
two complex analytic manifolds of dimensions $d_X$ and $d_Y$, respectively.

Consider the usual transfer-module 
$\shd_{X\to Y}\eqdot \sho_X\tens[f^{-1}\sho_Y] f^{-1}\shd_Y$ with its structure of 
$(\shd_X,f^{-1}\shd_Y)$-bimodule.  
Denote by $\doim f$ and $\deim f$ the functors of direct image and proper direct image in the $\shd$-modules framework.

Set 
\eqn
\shd_{X\to Y,\hbar}\eqdot \varprojlim_n (\sho_X\tens[f^{-1}\sho_Y] f^{-1}(\shd_Y^\h/\h^{n+1}\shd_Y^\h)).
\eneqn
Since each component of the projective limit has a natural structure of $\h$-torsion $(\shd_X^\hbar,f^{-1}\shd_Y^\hbar)$-bimodule, 
then $\shd_{X\to Y, \hbar}$ is also a $(\shd_X^\hbar,f^{-1}\shd_Y^\hbar)$-bimodule. 
Hereafter, when there is no risk of confusion, we use the following abbreviations: 
$\shk_\h\eqdot\shd_{X\to Y,\h}$ and $\shk\eqdot\shd_{X\to Y}$.

Note that $\shk^\hbar$ is also a $(\shd_X^\hbar,f^{-1}\shd_Y^\hbar)$-bimodule (cf. Remark~\ref{remark1}) and one
has an isomorphism of bimodules $\shk_\h \simeq \shk^\hbar$. 
Moreover, $\shk_\h\simeq\shk^\rhb$ is cohomologically complete, free of $\hbar$-torsion and $\gr(\shk_\h)\simeq\shk$ 
(Proposition 4.5 of~\cite{MMFR}).

%\begin{proof}
%Recall that $\shk$ is flat over $\sho_X$, since $f^{-1}\shd_Y$ is a free $f^{-1}\sho_Y$-module.
%Therefore (a) is a direct consequence of Theorem~\ref{flatness criteria}.

%It is well known that $\shk$ is a coherent left $\shd_X$-module in the smooth case, thus (b) follows by Theorem~\ref{T12}.
%\end{proof}

The direct image functor and the proper direct image functor in the framework 
of $\shd^\hbar$-modules are denoted respectively by $\doimh f$ and $\deimh f$ and defined by: 
\eqn
&\doimh f&:\Derb(\shd_X^\hbar)\to \Derb(\shd_Y^\hbar), \quad
\shm\mapsto \roim f(\shm\lltens[\shd_X^\hbar]\shk_\h); \\
&\deimh f&:\Derb(\shd_X^\hbar)\to \Derb(\shd_Y^\hbar), \quad 
\shm\mapsto \reim f(\shm\lltens[\shd_X^\hbar]\shk_\h).
\eneqn
Direct images for $\shd^\hbar$-modules are introduced in~\cite{MMFR} 
(where the transfer module in the $\h$-setting is simply denoted by $\mathcal{K}$).

Clearly, for $\shm\in\Derb(\shd_X^\hbar)$ the isomorphisms $\gr(\doimh f(\shm))\simeq \doim f(\gr(\shm))$ 
and $\gr(\deimh f(\shm))\simeq\deim f(\gr(\shm))$ hold.

%The next result is an easy consequence of the properties of $\gr$.

%\begin{lemma}
%For any $\shm\in\Derb(\shd_X^\hbar)$, the isomorphisms $\gr(\doimh f(\shm))\simeq \doim f(\gr(\shm))$ 
%and $\gr(\deimh f(\shm))\simeq\deim f(\gr(\shm))$ hold in $\Derb(\shd_Y)$. 
%\end{lemma}

%\begin{lemma}[\cite{MMFR}, Lemma 4.17]
%For any $\shm\in\Derb_\coh(\shd_X^\hbar)$, the direct image $\doimh f(\shm)$ is cohomologically complete.
%\end{lemma}

%\begin{theorem}[\cite{MMFR}, Theorem 4.18]\label{direct image}
%Assume that $\shm\in\Derb_\gd(\shd_X^\hbar)$ and that $f$ is proper when restricted to the support of $\shm$.  
%Then, $\doimh f(\shm)\in\Derb_\gd(\shd_Y^\hbar)$.
%\end{theorem}

\paragraph{$f$-characteristic variety.} 

Following~\cite{ScS}, the $f$-characteristic variety of $\shm\in\Derb_\coh(\shd_X)$, denoted by $\chv_f(\shm)$, 
is a closed conic analytic subvariety of $T^\ast X$ depending on $f$ and satisfying the formula
\eq\label{estimative}
\SSi(\shm\lltens[\shd_X]\shk)\subset \chv_f(\shm).
\eneq

Note that $\chv_{a_X}$ is the usual characteristic variety of $\shm$, denoted simply by $\chv(\shm)$.
In this case, the inclusion~\eqref{estimative} gives the well-known estimative: 
$\SSi(\Sol(\shm))\subset \chv(\shm)$. 
Here $\Sol$ denotes the solutions functor in the $\shd$-module framework, 
whose counterpart in the $\shd^\hbar$-modules framework is the functor $\Sol_\h$ studied in~\cite{DGS}: 
\eqn
\Sol_\h:\Derb_\coh(\shd_X^\hbar)^\rop\to \Derb(\coro_X), \quad \shm\mapsto \rhom[\Dh](\shm,\sho_X^\h).
\eneqn

\begin{definition}\label{D39}
The $f$-characteristic variety of $\shm\in\Derb_\coh(\Dh)$ is denoted by 
$\chv_{f,\h}(\shm)$ and defined by $\chv_{f,\h}(\shm)\eqdot\chv_f(\gr(\shm))$. 
\end{definition}

For any $\shm\in\Derb_\coh(\shd^\hbar_X)$, $\chv_{a_X,\h}(\shm)$ coincides 
with the characteristic variety of $\shm$ denoted by $\chv_\h(\shm)$.

%Recall that $\shm\in\Derb_\coh(\shd_X^\h)$ is said to be holonomic if $\gr(\shm)$ is holonomic in the $\shd$-modules sense, 
%that is, if $\chv_\h(\shm)=\chv(\gr(\shm))$ is a lagrangian subvariety of $T^\ast X$. 

\begin{lemma}\label{regularity}
For any $\shm\in\Derb_\coh(\Dh)$, we have $\SSi(\shm\lltens[\Dh] \shk_\h)\subset \chv_{f,\h}(\shm)$.
\end{lemma}
\begin{proof}
Since $\shm$ is coherent and $\shk_\h$ is cohomologically complete, 
$\shm\lltens[\Dh]\shk_\h$ is also cohomologically complete by Proposition 1.6.5 of~\cite{KS3}. 
Hence, $$\SSi(\shm\lltens[\Dh]\shk_\h)=\SSi(\gr(\shm)\lltens[\shd_X]\shk).$$
Therefore, the result follows from estimative~\eqref{estimative}. 
\end{proof}

We remark that $\SSi(\Sol_\h(\shm))=\chv_\h(\shm)$ is already proved in~\cite{DGS}.

\paragraph{Elliptic pairs.}

\begin{definition}\label{D48}
A pair $(\shm,F)$ with
$\shm\in\Derb_\coh({\shd_X^\h}^\rop)$ and $F\in\Derb_\Rc(\coro_X)$ 
is an $f$-elliptic pair over $\coro$ if
$\chv_{f,\h}(\shm)\cap \SSi(F)\subset T^\ast_X X$.
If in addition $\shm\in\Derb_\gd({\shd_X^\h}^\rop)$, then $(\shm,F)$ is said a good $f$-elliptic pair over $\coro$.
The support of the pair $(\shm,F)$ is the intersection $\supp(\shm) \cap \supp(F)$.
\end{definition}

Since $\gr$ preserves the micro-support of $\R$-constructible 
sheaves and the characteristic variety of coherent $\shd^\h$-modules we have:

\begin{proposition}\label{P1}   
A pair $(\shm,F)$ is an $f$-elliptic pair over $\coro$ 
if and only if $(\gr(\shm),\gr(F))$ 
is an $f$-elliptic pair over $\C$ (in the sense of~\cite{ScS}). 
\end{proposition}

If $(\shm,F)$ is an $f$-elliptic pair over $\coro$, then $(\ul{\RD}_{\h,X}\shm,\RDD_{\coro_X} F)$ is also 
an $f$-elliptic pair over $\coro$, the dual $f$-elliptic pair of $(\shm,F)$. 

Assume that $X$ is the complexification of a real analytic manifold $M$. 

\begin{definition}\label{def:elliptic}
We say that a coherent $\shd_X^\h$-module $\shm$ is an elliptic $\Dh$-module if 
$(\shm,\coro_M)$ is an $a_X$-elliptic pair over $\coro$. 
We say that an operator $P\in\shd_X^\h$ is an elliptic operator if $\Dh/ \Dh P$ is an elliptic $\Dh$-module.
\end{definition}

In other words, $\shm\in\mdc[\shd_X^\h]$ is elliptic if $\chv_\h(\shm)\cap T^\ast_M X\subset T^\ast_X X$. 
Moreover, one deduces from Lemma 3.5 of~\cite{DGS} that $P\in\Dh$ is elliptic 
if and only if it is locally written as $P=P_0+\h P'$ for some $P'\in\shd_X^\h$ and 
$P_0$ an elliptic operator in the classical sense. 
Take a system of holomorphic coordinates $(x;\eta)$ on $T^\ast X$. 
$P_0$ is elliptic if its principal symbol satisfies $\sigma(P_0)((x;i\eta))\neq 0$ for $\eta\neq 0$.
Take, for example, $X=\C^n$, $M=\R^n$ and denote by $\Delta$ 
the Laplace operator. 
Then, $P=\Delta+\h P^\prime$ is elliptic for any $P^\prime\in\Dh$.

The meaning of elliptic pairs on the real analytic setting 
illustrates why one can regard the theory of elliptic pairs (over $\coro$) 
as a natural generalization of the theory of elliptic systems on real analytic manifolds.

\section{Theorems on elliptic pairs over $\C[[\h]]$}

\paragraph{Regularity theorem.}

\begin{theorem}\label{T53}
Let $(\shm,F)$ be an elliptic pair over $\coro$. 
Then, the natural morphism below is an isomorphism in $\Derb(\coro_X)$:
\eq\label{eq:regularity1}
F\lltens[\coro_X] (\shm\lltens[\shd_X^\h]\shk_\h) \to \rhom[\coro_X](\RDD_{\coro_X} F, \shm\lltens[\shd_X^\h]\shk_\h)
\eneq
\end{theorem}
\begin{proof} 
The morphism~\eqref{eq:regularity1} is induced by the isomorphism 
$F\simeq \RDD_{\coro_X} \RDD_{\coro_X} F$ and by the canonical morphism:
\eqn
\RDD_{\coro_X} \RDD_{\coro_X} F \lltens[\coro_X] (\shm\lltens[\shd_X^\h] \shk_\h) 
\to \rhom[\coro_X] (\RDD_{\coro_X} F, \shm\lltens[\shd_X^\h]\shk_\h).
\eneqn
Note that $\RDD_{\coro_X} F$ has $\R$-constructible cohomology and $\SSi(\RDD_{\coro_X} F)=\SSi(F)^a$.
The transversality condition on the pair $(\shm,F)$ together with Lemma~\ref{regularity} entail:
\eqn
\SSi(\shm\lltens[\Dh]\shk_\h)\cap\SSi(\RDD_{\coro_X} F)^a \subset T^\ast_X X.
\eneqn
The conclusion follows by Proposition 5.4.14 of~\cite{KS1}.
\end{proof}

\begin{remark}\label{RT53} 
The isomorphism~\eqref{eq:regularity1} is written in the following form in terms of left $\shd_X^\h$-modules: 
\eqn
\rhom[\Dh](\shm,\RDD_{\coro} F \lltens[\coro_X]\shk_\h)
\isoto  \rhom[\Dh](\shm,\rhom[\coro_X](F,\shk_\h)).
\eneqn 
\end{remark}

We want to refine the regularity property in the case where $F$ 
is the formal extension of an object in $\Derb_\Rc(\C_X)$.
Let us start with some auxiliar results.

\begin{lemma}\label{L34}   
Let $F,G\in \md[\C_X]$.
There is a natural morphism in $\md[\coro_X]$:
\eq\label{morphism1}
F \tens G^\h\to (F\tens G)^\h.
\eneq
\end{lemma}
\begin{proof} 
It is enough to note that there is a projective system of morphisms
$F\tens G^\h\to F\tens (G\tens(\coro/\h^{n+1}\coro))
\isoto(F\tens G)\tens(\coro/\h^{n+1}\coro)$. 
\end{proof}

\begin{lemma}\label{L35} 
Let $F,G\in \Derb(\C_X)$. 
Then there is a natural bifunctorial morphism in $\Derb(\coro_X)$: 
\eq\label{eq:regularity3}
F \tens G^\rhb \to (F \tens G)^\rhb.
\eneq
\end{lemma}  
\begin{proof}
(i) First, fix $F\in\md[\C_X]$ and consider the two functors
from $\md[\C_X]$ to $\md[\coro_X]$: $\theta_1\cl G\mapsto F\tens G^\h$ and $\theta_2\cl G\mapsto (F\tens G)^\h$.
By Lemma~\ref{L34}, there is a morphism of functors $u\cl\theta_1\to \theta_2$.
Since both functors are left exact, 
this morphism $u$ extends to the
derived functors and we get the morphism~\eqref{eq:regularity3} 
for a fixed $F\in\mdrc[\C_X]$.

\noindent
(ii) Now let us fix $G\in \Derb(\C_X)$ and 
consider the two functors from $\md[\C_X]$ to $\Derb(\coro_X)$: $\lambda_1\cl F\mapsto F\tens G^\rhb$ 
and $\lambda_2\cl F\mapsto (F\tens G)^\rhb$.
By (i) there exists a morphism of functors $v\cl\lambda_1\to\lambda_2$.
Both functors extend naturally to the category of bounded complexes 
$\rmC^\Rb(\md[\C_X])$ and send complexes quasi-isomorphic to zero to
objects isomorphic to zero in $\Derb(\coro_X)$. Hence, both functors
extend to $\Derb(\C_X)$ as well as the morphism of functors $v$.
\end{proof}

\begin{lemma}\label{preparation} 
For $F,G\in\Derb(\C_X)$, we have a commutative diagram in $\Derb(\coro_X)$:
\eq\label{diagram1}
&&\xymatrix{
\RDD F \tens G^\rhb \ar[r]\ar[dr]
                        &(\RDD  F\tens G)^\rhb\ar[d]\\
&\rhom(F,G)^\rhb,
}\eneq
such that the morphism in the horizontal arrow is the one given by Lemma~\ref{L35}.
\end{lemma}
\begin{proof}
The oblique arrow is the composition of two canonical morphisms:
\eqn
\rhom(F,\C)\tens G^\rhb\to\rhom(F,G^\rhb)
\isoto\rhom(F,G)^\rhb. 
\eneqn
The vertical arrow is defined by applying the functor $(\scbul)^\rhb$ 
to the canonical morphism $\rhom(F,\C)\tens G\to\rhom(F,G)$. 
The commutativity of the 
diagram is obvious.
\end{proof}

\begin{remark}
Assuming that $G\in\md[\shd_X^\h]$ in Lemma~\ref{L34}, then the morphism~\eqref{morphism1} is $\shd_X^\h$-linear. 
Similarly, if $G\in\Derb(\shd_X^\h)$, then the morphism~\eqref{eq:regularity3} is a morphism in $\Derb(\shd_X^\h)$ and 
the diagram in Lemma~\ref{preparation} is a commutative diagram in $\Derb(\shd_X^\h)$. 
In the sequel we shall use such diagram in the case $G= \shk_\h$.
\end{remark}

\begin{theorem}\label{T54} 
Let $F\in\Derb_\Rc(\C_X)$ and suppose that $(\shm,F^\h)$ is an $f$-elliptic pair over $\coro$. 
Then there is a commutative 
diagram of isomorphisms in $\Derb(\coro_X)$:
\eq\label{regularity7}
\xymatrix{
\rhom[\Dh](\shm,\RDD_{\coro_X} F^\h\lltens[\coro_X] \shk_\h)\ar[r]^-{\sim}\ar[dr]^-\sim
                        &\rhom[\Dh](\shm,(\RDD F\tens \shk)^\rhb)\ar[d]^-\sim\\
&\rhom[\Dh](\shm, \rhom(F,\shk)^\rhb).
}\eneq
\end{theorem}

\begin{proof}
First note that we obtain the diagram~\eqref{regularity7} by applying the functor 
$\rhom[\Dh](\shm,\scbul)$ to the commutative diagram provided by Lemma~\ref{preparation}:
\eqn
&&\xymatrix{
\RDD F\tens \shk_\h \ar[r]\ar[dr]
                        &(\RDD F\tens \shk)^\rhb\ar[d]\\
&\rhom(F,\shk)^\rhb.
}\eneqn
We also use the fact that $\RDD_{\coro_X}F^\h\lltens[\coro_X] \shk_\h$
is isomorphic to 
$\RDD F\tens \shk_\h$.

Note that the oblique arrow is an isomorphism since it is the composition of two canonical isomorphisms: 
\eqn
\rhom[\Dh](\shm,\RDD_{\coro_X}F^\h\lltens[\coro_X] \shk_\h)
&\isoto& \rhom[\Dh](\shm,\rhom[\coro_X](F^\h,\shk_\h)) \\
&\isoto&\rhom[\Dh](\shm,\rhom(F,\shk_\h)^\rhb),
\eneqn
The first one results from applying the 
regularity theorem~\ref{T53} to the $f$-elliptic pair  
$(\shm,F^\h)$ (see also Remark~\ref{RT53}) and the second isomorphism results from Lemma~\ref{T31}. 

The vertical arrow is also an isomorphism in $\Derb(\coro_X)$. 
In fact, we conclude from~\cite[Prop. 1.5.10]{KS3} and~\cite[Prop. 2.2]{DGS}
that the objects on both sides of the vertical arrow are cohomologically 
complete. Hence, since $\gr$ is conservative on cohomologically complete objects, 
it is enough to prove that the canonical morphism below is an isomorphism:
\eq\label{regularity6}
\rhom[\D](\gr(\shm),\RDD F \tens \shk) \to
\rhom[\D](\gr(\shm),\rhom(F,\shk)).
\eneq
In fact, $(\gr(\shm),F)$ is an $f$-elliptic pair over $\C$ 
and the morphism~\eqref{regularity} is precisely the regularity isomorphism 
for elliptic pairs over $\C$ applied to $(\gr(\shm),F)$ (see Theorem 2.15 of~\cite{ScS}).

We have proved that the oblique and vertical arrows in diagram~\eqref{regularity7} are isomorphisms. 
The commutativity of the diagram allow us to conclude that the horizontal arrow is also an isomorphism. 
\end{proof}

Using the functorial properties of $(\scbul)^\rhb$ we also get: 

\begin{corollary}
Let $(\shm,F)$ be an $f$-elliptic pair over $\C$. 
Then, $(\shm^\h,F^\h)$ is an $f$-elliptic pair over $\coro$ and 
there are canonical isomorphisms in $\Derb(\coro_X)$:
\eqn
\rhom[\shd_X^\h](\shm^\h,\RDD_{\coro_X} F^\h\lltens[\coro_X] \shk_\h) &\simeq&
\rhom[\shd_X](\shm,\RDD  F \tens\shk)^\rhb \\
&\simeq&\rhom[\shd_X](\shm,\rhom[\C_X](F,\shk))^\rhb.
\eneqn
\end{corollary}

\paragraph{Finiteness theorem.}

%The following theorem is a stronger version of the finiteness criteria for direct images of $\Dh$-modules proved in~\cite{MMFR}.

\begin{theorem}\label{T55} 
Let $(\shm,F)$ be a good $f$-elliptic pair over $\coro$ 
and suppose that $f$ is proper when restricted to the support of $(\shm,F)$.
Then, $\deimh f(\shm\lltens[\coro_X]F)$ is an object of $\Derb_\gd(\shd_Y^\h)$.
\end{theorem}
\begin{proof}
Note that by Theorem 1.6.4 of~\cite{KS3} it suffices to prove that: 
\bnum
\item $\deimh f(F\lltens[\coro_X] \shm)$ is cohomologically complete;
\item $\gr(\deimh f(F\lltens[\coro_X] \shm))$ is an object of $\Derb_\gd(\shd_Y^\h)$.
\enum 
Set $\shl\eqdot (F\lltens[\coro_X] \shm) \lltens[\shd_X^\hbar] \shk_\h$.
The regularity theorem~\ref{T53} yields the isomorphism:
\eqn
\shl \simeq \rhom[\Dh](\RDD_{\coro_X} F,\shm\lltens[\Dh]\shk_\h).
\eneqn
Since $\shm\lltens[\Dh]\shk_\h$ is cohomologically complete, $\shl$ is also cohomologically 
complete in view of Propositions 1.5.10 of~\cite{KS3}. 
Finally, $\deimh f(F\lltens[\coro_X] \shm)=\reim f(\shl)$ is cohomologically complete by Proposition~1.5.12 of~\cite{KS3} and 
since $\roim f$ and $\reim f$ coincide by the hypothesis on the support. 

On the other hand, note that $(\gr(\shm),\gr(F))$ is an $f$-elliptic pair over $\C$ 
that satisfies the conditions of the finiteness theorem for elliptic pairs over $\C$. 
Hence, $\gr(\deimh f(F\lltens[\coro_X] \shm))\simeq \deim f(\gr(F)\tens \gr(\shm))$ 
is an object of $\Derb_\gd(\shd_Y)$ by Theorem 4.2. of~\cite{ScS}.
\end{proof}

\paragraph{Duality theorem.}

The results in this paragraph are obtained assuming that $f:X\to Y$ is a smooth morphism of complex manifolds. 

\begin{proposition}\label{prop duality2}
If $f:X\to Y$ is smooth, then $\opb f \shd^\h_Y \simeq (\opb f \shd_Y)^\h$.
\end{proposition}
\begin{proof}
This is a particular case of Proposition~\ref{VB} choosing for $\mathcal{B}$ the family of compact Stein subsets of $Y$.
\end{proof}

The rings $\opb f \shd^\h_Y$ and $(\opb f \shd_Y)^\h$ are not isomorphic in general: 
if $i:Y\into X$ is the embedding of a closed submanifold of dimension $d_Y<d_X$, 
then $\opb i \shd^\h_X \to (\opb i \shd_X)^\h$ is a monomorphism which is not an epimorphism.

Note that in the smooth case $\shk_\h$ is a coherent left $\shd_X^\h$-module.
Hence, for any coherent right $\shd_X^\h$-module $\shm$, one has:
\eq\label{aux4}
\shm\tens[\shd_X^\h]\shk^\h \simeq \shm\tens[\shd_X^\h] (\shd_X^\h\tens[\shd_X]\shk)\simeq \shm\tens[\shd_X]\shk.
\eneq
By the isomorphisms in~\eqref{aux4}, the object $\shm\tens[\shd_X]\shk$ has a structure of $\opb f({\shd_Y^\h}^\rop)$-module. 
Passing to the derived category and applying $\reim f$, one concludes that $\deim f(\shm)$ is an object of 
$\Derb({\shd_Y^\h}^\rop)$ and that it is isomorphic to $\deimh f(\shm)$ for any $\shm\in\Derb_\coh({\shd_X^\h}^\rop)$. 
In other words, the direct image of $\shm$ in the $\h$-setting coincides with 
its direct image as a $\shd$-module. We use this fact in the sequel.

%In particular, the right hand side of~\eqref{aux4} becomes a $\opb f({\shd_Y^\h}^\rop)$-module. 
%We extend this argument to the derived category to conclude that, for $\shm\in\Derb({\shd_Y^\h}^\rop)$, 
%$\deim f(\shm)$ is as an object of $\Derb({\shd_Y^\h}^\rop)$ and it is isomorphic to $\deimh f(\shm)$. 

%\begin{example}\label{mono}
%Let $i:Y\into X$ be the embedding of a closed submanifold of dimension $d_Y<d_X$. 
%Let $V$ be an open subset of $Y$.
%A section $Q\in\sect(V;\opb i(\shd_X^\h))$ can be regarded as a formal series  
%$Q=\sum Q_n|_V \h^n$, with $Q_n$ in $\sect(U;\shd_X)$ for some fixed open subset $U$ of $X$ verifying $U\cap Y=V$.
%On the other hand, a section $R\in\sect(V;(\opb i\shd_X)^\h)$ can be regarded as 
%a formal series $R=\sum R_n|_V \h^n$, where each $R_n$ is a section of $\sect(U_n;\shd_X)$ for some
%open subset $U_n$ of $X$ verifying $U_n \cap Y=V$.
%Hence, the canonical morphism $\opb i \shd^\h_X \to (\opb i \shd_X)^\h$ is a monomorphism which is not epimorphism.
%\end{example}

\begin{lemma}\label{Aux3}
Let $\shm\in\Derb_\coh(\shd_X^\rop)$. 
There is a morphism $\deimh f(\shm^\h)\to \deim f(\shm)^\rhb$ in $\Derb({\shd_Y^\h}^\rop)$.
\end{lemma}
\begin{proof}
First note that we have a chain of isomorphisms in $\Derb(\opb f{\shd_Y^\h}^\rop)$:
\eqn
\shm^\h\lltens[\shd_X]\shk&\simeq& \shm^\h \lltens[\shd_X] \RDD_{X}\RDD_{X}\shk  \\
&\simeq& \rhom[\shd_X](\RDD_{X} \shk,\shm_\h) \\
&\simeq& \rhom[\shd_X](\RDD_X \shk, \shm)^\rhb. 
\eneqn
The first and second isomorphisms follow from the coherence of $\shk$ over $\shd_X$. 
The third follows from the properties of $(\scbul)^\rhb$ (Lemma 2.3 of~\cite{DGS}). 

Applying $\reim f$ we get the following isomorphism in $\Derb(\shd_Y^\rop)$: 
\eqn 
\deimh f(\shm^\h)\simeq \reim f((\shm\lltens[\shd_X]\shk)^\rhb).
\eneqn 
Finally, Lemma~\ref{Aux2} entails the following morphism in $\Derb({\shd_Y^\h}^\rop)$: 
\eqn 
\deimh f(\shm^\h)\to \reim f((\shm\lltens[\shd_X]\shk))^\rhb = \deim f(\shm)^\rhb.
\eneqn
\end{proof}

If $\shm$ is a right $(\shd_X,\shd_X)$-bimodule, its direct image as a bimodule is the  
object in the derived category $\Derb(\shd_Y^\rop\tens\shd_Y^\rop)$ defined by:
\eqn
\bideim f(\shm):=\reim f((\shm\lltens[\shd_X] \shk) \lltens[\shd_X] \shk).
\eneqn
In particular $\Omega_X\tens[\sho_X]\shd_X$ is a right $(\shd_X,\shd_X)$-bimodule and 
one has the so-called differential trace morphism in $\Derb(\shd_Y^\rop\tens\shd_Y^\rop)$:
\eq\label{trace1}
\tr_f:\bideim f(\Omega_X[d_X]\tens[\sho_X]\shd_X)\to \Omega_Y[d_Y]\tens[\sho_Y]\shd_Y.
\eneq

Note that if $\shm$ is a right $(\shd_X^\hbar,\shd_X^\hbar)$-bimodule then 
$\bideimh f(\shm)\eqdot \bideim f(\shm)$ is an object of $\Derb({\shd_Y^\hbar}^{\rop} \tens[\coro] {\shd_Y^\hbar}^{\rop})$.
 
\begin{proposition}\label{trace}
The morphism $f$ induces a differential trace morphism in the derived category 
$\Derb({\shd_Y^\hbar}^\rop \tens[\C^\hbar] {\shd_Y^\hbar}^\rop)$:
\begin{eqnarray*}
\tr_{f,\h}:\bideimh f(\Omega_X[d_X]\tens[\sho_X] \shd_X^\hbar) 
\to \Omega_Y[d_Y]\tens[\sho_Y] \shd_Y^\hbar.
\end{eqnarray*}
\end{proposition}
\begin{proof} 
We have the following isomorphisms in $\Derb({\shd_Y^\hbar}^\rop \tens[\C^\hbar] {\shd_Y^\hbar}^\rop)$:
\eqn
\bideimh f(\Omega_X[d_X] \tens[\sho_X]\shd_X^\hbar) &\simeq& \deimh f(\Omega_X[d_X] \tens[\sho_X] \shk_\h) \\ 
&\simeq& \deimh f((\Omega_X[d_X]\tens[\sho_X]\shk)^\h)
\eneqn
The first one results from the associativity of tensor products.
The second one results from $\Omega_X[d_X]\tens[\sho_X]\shk_\h\simeq (\Omega_X[d_X]\tens[\sho_X]\shk)^\h$. 

The morphism $\tr_{f,\h}$ is then constructed composing the morphisms below:
%following morphisms in $\Derb({\shd_Y^\hbar}^\rop \tens[\C^\hbar] {\shd_Y^\hbar}^\rop)$:
\eqn
\deimh f((\Omega_X[d_X]\tens[\sho_X] \shk)^\h)
\to \deim f(\Omega_X[d_X]\tens[\sho_X] \shk)^\rhb 
\to  \Omega_Y[d_Y]\tens[\sho_Y]\shd_Y^\h.
\eneqn
The first morphism is an application of Lemma~\ref{Aux3} which, in this case, preserves the bimodule structures involved. 
The second morphism is the formal extension of the classical trace morphism~\eqref{trace1}.
\end{proof}

\begin{remark}
Consider the absolute case $f=a_X$. 
In this case one gets $\shk_\h\simeq \sho_X^\h$ and the following isomorphisms hold in $\Derb(\coro)$:
\eqn 
\bideimh{a_X}(\Omega_X[d_X]\tens[\sho_X]\shd_X^\h) 
\simeq \rsect_c(X;\coro_X[2d_X])\simeq\rsect_c(X;\omega_X^\h).
\eneqn
Thus, the trace morphism $\tr_{a_X,\h}$ coincides with the morphism  
$\rsect_c(X,\omega_X^\h) \to \coro$ induced 
by the natural transformation $\reim{a_X} \epb{a_X} \to \id$.
\end{remark}

\begin{lemma}\label{P57}
Let $\shm\in\Derb({\shd_X^\h}^\rop)$. We have a canonical morphism in $\Derb({\shd_Y^\h}^\rop)$:
\eq\label{duality1}
\deimh f(\ul{\RD}_{\h,X}\shm) \to \ul{\RD}_{\h,Y}( \deimh f \shm).
\eneq
\end{lemma}
\begin{proof} 
First note that there is a basis change morphism in $\Derb({\shd_Y^\h}^\rop)$:
\eqn
&&\deimh f(\ul{\RD}_{\h,X}(\shm)) 
\simeq\reim f(\rhom[\Dh](\shm, \Omega_X[d_X]\tens[\sho_X]\shk_\h))\\
&&\to
\rhom[\shd_Y^\h](\reim f(\shm\lltens[\shd_X]\shk),\reim f((\Omega_X[d_X]\tens[\sho_X]\shk_\h)\lltens[\shd_X]\shk))\\
&&\simeq\rhom[\shd_Y^\h](\deimh f(\shm), \bideimh f(\Omega_X[d_X]\tens[\sho_X]\shd_X^h)).
\eneqn
We obtain~\eqref{duality1} composing the above morphism with the morphism 
induced by Proposition~\ref{trace} on the second term of the $\rhom$.
\end{proof}

\begin{corollary}\label{C59} 
For $\shm\in\Derb({\shd_X^\h}^\rop)$ and $F\in\Derb(\coro_X)$, 
there is a canonical morphism in $\Derb({\shd_Y^\h}^\rop)$:
\eq\label{duality7}
\deimh f(\RDD_{\coro_X} F \lltens[\coro_X]\ul{\RD}_{\h,X}(\shm)) 
\to \ul{\RD}_{\h,Y}( \deimh f(F\lltens[\coro_X] \shm)).
\eneq
\end{corollary}
\begin{proof}
Note that there is a canonical morphism in $\Derb({\shd_X^\h}^\rop)$: 
\eq\label{duality6}
\RDD_{\coro_X} F \lltens[\coro_X]\ul{\RD}_{\h,X}(\shm)\to \ul{\RD}_{\h,X}(F\lltens[\coro_X] \shm).
\eneq
%The $\shd_X^\h$-module structures are induced by $\shm$.
The morphism~\eqref{duality7} is then obtained as a composition of morphisms:
\eqn
\deimh f(\RDD_{\coro_X} F \lltens[\coro_X]\ul{\RD}_{\h,X}(\shm)) 
&\to& \deimh f(\ul{\RD}_{\h,X}(F \lltens[\shd_X^\h]\shm))\\
&\to&\ul{\RD}_{\h,Y}( \deimh f(F\lltens[\coro_X] \shm)),
\eneqn
where the first arrow results from~\eqref{duality6} and the second arrow 
results from~\eqref{duality1}.
\end{proof}

\begin{theorem}\label{T60} 
Let $(\shm,F)$ be a good $f$-elliptic pair 
over $\coro$ and suppose that $f$ is proper when restricted to the support of $(\shm,F)$. 
Then, the canonical morphism
\eq\label{duality3}
\deimh f(\RDD_{\coro_X} F \lltens[\coro_X]\ul{\RD}_{\h,X}(\shm)) 
\to \ul{\RD}_{\h,Y}( \deimh f(F\lltens[\coro_X] \shm))
\eneq
is an isomorphism in $\Derb_\gd({\shd_Y^\h}^\rop)$.
\end{theorem}
\begin{proof} 
Set $\shl_1=\deimh f(\RDD_{\coro_X} F \lltens[\coro_X]\ul{\RD}_{\h,X}(\shm))$ and 
$\shl_2=\ul{\RD}_{\h,Y}( \deimh f(F\lltens[\coro_X] \shm))$. 

First note that both $(\shm,F)$ and $(\ul{\RD}_{\h,X}\shm, \RDD_{\coro_X} F)$ are $f$-elliptic pairs that 
verify the conditions of the finiteness theorem~\ref{T55}. 
Therefore $\shl_1$ and $\shl_2$ are objects of $\Derb_\gd({\shd_Y^\h}^\rop)$.
Hence, by the conservativism of $\gr$ on coherent objects, it suffices to check that 
the induced morphism $\gr(\shl_1)\to \gr(\shl_2)$ is an isomorphism in $\Derb_\gd(\shd_Y^\rop)$.
Consider the isomorphisms: 
\eqn
&&\gr(\shl_1)\simeq \deim f(\RDD_{\C_X} \gr(F) \tens \ul{\RD}_{X}(\gr(\shm)))\\
&&\gr(\shl_2)\simeq \ul{\RD}_{Y}( \deim f(\gr(F)\tens[\C_X] \gr(\shm))).
\eneqn
By the construction of the duality morphism, $\gr(\shl_1)\to\gr(\shl_2)$ 
is precisely the duality morphism for $f$-elliptic pairs over $\C$ applied to $(\gr\shm,\gr F)$.
We conclude that $\gr(\shl_1)\to\gr(\shl_2)$ is an isomorphism in $\Derb_\gd({\shd_Y}^\rop)$ 
since the pair $(\gr \shm, \gr F)$ is in the conditions of the duality theorem of~\cite{ScS} 
(see~\cite[Theorem~5.15]{ScS}). 
\end{proof}

\paragraph{Remarks on particular cases.} 
We can apply our theorems in some particular cases. 
This is similar to what is done in~\cite{ScS} and we leave out the details.

{\rm (i)} For each $\shm\in\Derb_\coh({\shd_X^\h}^\rop)$ the pair $(\shm, \coro_X)$ forms an $f$-elliptic pair over $\coro$. 
Theorems~\ref{T55} and~\ref{T60} in this particular case give finiteness and duality theorems (in the smooth case) for $\shd_X^\h$-modules. 

{\rm (ii)} The classical finiteness theorem for coherent $\sho_X$-modules (Grauert's theorem) 
can be generalized to the $\h$-framework using similar arguments to those employed in the proof of Theorem~\ref{T55}. 
A relative duality theorem for coherent $\Oh$-modules can also be proved under the condition $\opb f \sho_Y^\h\isoto (\opb f \sho_Y)^\h$. 
In fact, with this condition one is able to construct a duality morphism using the one from the $\sho_X$-modules theory. 
Again the general idea is similar (in fact it is easier) to that of Theorem~\ref{T60}, thus we don't give further details.
Let us just remark that the isomorphism $\opb f \sho_Y^\h\isoto (\opb f \sho_Y)^\h$ holds if $f$ is smooth 
(in Proposition~\ref{VB}, choose for $\mathcal{B}$ the family of open Stein subsets of $Y$).

On the other hand, we can apply our theorems to elliptic pairs of the form $(\shf\tens[\sho_X^\h]\shd_X^\h,\coro_X)$, $\shf$ being an object of $\Derb_\coh(\sho_X^\h)$. However, unlike the $\shd$-modules case in~\cite{ScS}, it is not clear if we obtain the finiteness and duality results for $\sho_X^\h$-modules mentioned above. In other words, we don't know what is the relation bewteen $\deimh f(\shf\tens[\sho_X^\h]\shd_X^\h)$ and $\reim f(\shf)\tens[\sho_Y^\h]\shd_Y^\h$. Only in the case $f=a_X$ it becomes obvious these objects are isomorphic. 

Note that in the case $f=a_X$ one gets absolute finiteness and duality results for $\sho_X^\h$-modules which are contained in the finiteness and duality results for deformation-quantization modules of~\cite{KS3}.

{\rm (iii)} Consider the case $f=a_X$. Let $(\shm,F)$ be a good $a_X$-elliptic pair with compact support. 
Then, Theorem~\ref{T55} says that the cohomology modules of the complex of global solutions of $\shm$ on the generalized 
sheaf of holomorphic functions associated to $F$ are finitely generated over $\coro$. 
In particular, we get the following statement: 
if $\shm$ is a good $\shd_X^\h$-module with compact support, 
then the cohomologies of $\rsect(X;\Sol_\h(\shm))$ are finitely generated over $\coro$.

{\rm (iv)} Let $X$ be the complexification of a real analytic manifold $M$ 
and let $\sha_M$ denote the sheaf of real analytic functions on $M$. 
Consider the following sheaves of real analytic functions with $\h$-parameter:
\eqn
\sha_{M,\h}\eqdot\C_M^\h\tens[\coro_X]\sho_X^\h\simeq\C_M\tens\sho_X^\h, \quad \sha_M^\h\eqdot(\C_M\tens\sho_X)^\h.
\eneqn
One has the isomorphism $\sha_M^\h\simeq\sha_M^\rhb$, 
since $\sha_M$ verifies the hypothesis of Proposition~\ref{chc5} taking for $\mathcal{B}$ the family of all open subsets of $M$.
Hence, both $\sha_{M,\h}$ and $\sha_M^\h$ are concentrated in degree 0 
and we can identify them to usual sheaves. 
There is a natural monomorpism $\sha_{M,\h}\into \sha_M^\h$ in $\md[\shd_X^\h]$ which is not an epimorphism
(this morphism is a particular case of Lemma~\ref{L35}).

Let us consider the c-soft sheaf of hyperfunctions on $M$, defined by $\shb_M\eqdot\rhom[\C_X](\C_M,\sho_X)$. 
Since $c$-soft objects are $\sect(K;\_)$-injective for any compact $K\subset X$, $\shb_M$ is $\h$-acyclic by Proposition~\ref{chc5}.
Hence: $$\shb_M^\rhb\simeq \shb_M^\h\simeq \rhom[\coro_X](\RDD_{\coro_X} \coro_M,\Oh).$$

%As a particular case of Lemma~\ref{preparation}, 
%we have the following commutative diagram of morphisms in $\md[\shd_X^\h]$:
%\eqn
%&&\xymatrix{
%\sha_{M,\h} \ar[r]\ar[dr] 
%&\sha_M^\h\ar[d]\\
%&\shb_M^{\h}
%}\eneqn

Applying our theorems to $a_X$-elliptic pairs of the form $(\shm,\coro_M)$ we obtain:
 
%Let $\shm$ be a coherent $\Dh$-module. 
%Recall that we say that $\shm$ is elliptic on $M$ if 
%$(\shm,\coro_M)$ is an elliptic pair over $\coro$, cf. Definition~\ref{def:elliptic}.
%Applying our main results to elliptic pairs of the form $(\shm,\coro_M)$, 
%we obtain the following regularity, finiteness and 
%duality property for elliptic $\shd_X^\h$-modules. 

\begin{corollary}\label{C63}
Let $\shm$ be an elliptic $\Dh$-module on $M$. 
\banum 
\item There is a commutative diagram of isomorphisms 
in $\Derb(\coro_X)$:
\eqn
&&\xymatrix{
\rhom[\Dh](\shm,\sha_{M,\h}) \ar[r]^-\sim\ar[dr]^-\sim 
&\rhom[\Dh](\shm,\sha_M^\h)\ar[d]^\sim\\
&\rhom[\Dh](\shm,\shb_M^{\h}).
}\eneqn
\item
If $M$ is compact and $\shm \in \Derb_\gd(\Dh)$, then $\rsect(M;\Omega_M^\h\lltens[\shd_M^\h]\shm)$
is the dual of $\rsect(M;\rhom[\shd_M^\h](\shm,\shb_{M,\h}))$ and both objects 
belong to $\Derb_f(\coro)$.
\eanum
\end{corollary}

\small

\providecommand{\bysame}{\leavevmode\hbox to3em{\hrulefill}\thinspace}

\noindent
David Raimundo, Centro de Matem\'atica e Aplica\c c\~oes Fundamentais da Universidade de Lisboa,
Departamento de Matem\'atica, FCUL, Edif\'icio C6, piso 2, Campo Grande
1749-16, Lisboa, Portugal, \texttt{dsraimundo@fc.ul.pt}

\end{document}